\documentclass[10pt]{amsart}
\usepackage[margin=1.2in,marginparsep=0.1in,marginparwidth=1in]{geometry}
\usepackage{amssymb,amsmath,amsthm,amstext,amscd,latexsym,graphics,graphicx,bbm,caption}
\usepackage[dvipsnames,svgnames,table]{xcolor}
\usepackage[plainpages=false,colorlinks=true, pagebackref]{hyperref}

\usepackage{mathtools}
\usepackage{tikz-cd}
\usepackage{cleveref}
\usepackage{enumitem}
\usepackage{setspace}
\usepackage{centernot}

\hypersetup{citecolor=Sepia,linkcolor=blue, urlcolor=blue}

\usepackage{cite}  

\makeatletter
\def\@tocline#1#2#3#4#5#6#7{\relax
	\ifnum #1>\c@tocdepth % then omit
	\else
	\par \addpenalty\@secpenalty\addvspace{#2}%
	\begingroup \hyphenpenalty\@M
	\@ifempty{#4}{%
		\@tempdima\csname r@tocindent\number#1\endcsname\relax
	}{%
		\@tempdima#4\relax
	}%
	\parindent\z@ \leftskip#3\relax \advance\leftskip\@tempdima\relax
	\rightskip\@pnumwidth plus4em \parfillskip-\@pnumwidth
	#5\leavevmode\hskip-\@tempdima
	\ifcase #1
	\or\or \hskip 2em \or \hskip 2em \else \hskip 3em \fi%
	#6\nobreak\relax
	\dotfill\hbox to\@pnumwidth{\@tocpagenum{#7}}\par
	\nobreak
	\endgroup
	\fi}
\makeatother
\newtheorem{intro-thm}{Theorem}[]
\theoremstyle{plain}
\newtheorem{thm}{Theorem}[section]
\newtheorem{theorem}[thm]{Theorem}

\newtheorem{lemma}[thm]{Lemma}

\newtheorem{proposition}[thm]{Proposition}

\theoremstyle{definition}

\newtheorem{definition}[thm]{Definition}
\newtheorem{example}[thm]{Example}

\newcommand{\Tr}{{\rm Tr}}
\newcommand{\Gal}{{\rm Gal}}

\newcommand{\A}{{\mathbb A}}

\newcommand{\Q}{{\mathbb Q}}

\newcommand{\Z}{{\mathbb Z}}
\newcommand{\p}{{\mathfrak p}}

\newenvironment{acknowledgements}
{\par\noindent\textbf{Acknowledgements.}\ }
{\par}
\newcommand{\Nm}{\mathrm{N}}

\input{xy}
\newcommand{\Br}{\operatorname{Br}}

\newcommand{\res}{\operatorname{Res}}
\newcommand{\inv}{\operatorname{inv}}

\usepackage[OT2,T1]{fontenc}
\DeclareSymbolFont{cyrletters}{OT2}{wncyr}{m}{n}
\DeclareMathSymbol{\Sha}{\mathalpha}{cyrletters}{"58}
\begin{document}
	\title[Vanishing of Brauer--Manin Obstruction under Field Extensions]{On the Vanishing of the Brauer-Manin Obstruction for Normic Bundles}
	\author[M. Biswas]{Mridul Biswas} \address{School of Mathematics and Statistics, University of Canterbury, Christchurch, 8041, New Zealand } \email{biswasmridul123@gmail.com}
	\author[D. C-Ramachandran]{Divyasree C-Ramachandran} \address{Department of Mathematics, Indian Institute of Science Education and Research Pune, Dr Homi Bhabha Rd, Pashan, Pune, 411008, India}\email{crdivya99@gmail.com}
    
	\author[B. Samanta]{Biswanath Samanta} \address{Department of Mathematics, Indian Institute of Science Education and Research Mohali, Sector 81, SAS Nagar, 140306, India}\email{biswa.uoh@gmail.com}
	
	\thanks{Keywords: Brauer group, Brauer-Manin obstruction, normic bundles, conic bundles, rational points}
\begin{abstract}

We study the behaviour of the Brauer--Manin obstruction under finite extensions for normic bundles. For $(p,mp)$-normic bundles, we establish vanishing results for the Brauer–Manin obstruction under divisibility conditions on the extension degree.
\end{abstract}

\begin{abstract}
We study the behaviour of the Brauer--Manin obstruction to the existence of rational points under finite field extensions. For \((p,mp)\)-normic bundles over number fields, we prove that the Brauer--Manin obstruction vanishes after base change to finite extensions whose degrees satisfy suitable \(p\)-divisibility conditions depending on \(m\). We further show that, for \((p,2p)\)-normic bundles with $p=2$ or $3$, it is enough to assume that the extension degree is divisible by \(p\). We also prove that the divisibility hypothesis is, in general optimal, by constructing a conic bundle for which the Brauer--Manin obstruction persists over a quadratic extension.
\end{abstract}

		\maketitle

%%%%%%%%%%%%%%%%%%%%%%%%%%%%%%%%%%%%%%%%%%%%%%%%%%%%%%%%%%%%%%%%%%%%%%%%%%%%%%%%%%%%%%%%%%%%%%%%%%%%%%%

%%%%%%%%%%%%%%%%%%%%%%%%%%%%%%%%%%%%%%%%%%%%%%%%%%%%%%%%%%%%%%%%%%%%%%%%%%%%%%%%%%%%%%%%%%%%%%%%%%%%%%%
\section{Introduction}\label{Introduction}

The Brauer--Manin obstruction plays a central role in the study of rational points on varieties over number fields. A fundamental question is whether a variety \(X\) admits a \(k\)-rational point. When \(X(k)=\emptyset\), it is natural to ask instead over which finite extensions \(L/k\) the variety \(X\) acquires a rational point, that is, for which extensions \(X(L)\neq\emptyset\). Since this is often difficult to determine directly, one instead considers the weaker question of when the Brauer--Manin set $X(\mathbb{A}_L)^{\mathrm{Br}}$ is non-empty.

For a smooth projective variety $X$ over a global field $k$, one has \(X(k) \subset X(\mathbb{A}_k)^{\mathrm{Br}} \subset X(\mathbb{A}_k)\), where $X(\mathbb{A}_k)$ denotes the adelic points of $X$, and $X(\mathbb{A}_k)^{\mathrm{Br}}$ the Brauer--Manin set. It is natural to study how these sets behave under finite field extensions.

In \cite[Section~4]{Vir}, the author asks how the sets $X(\mathbb{A}_k)^{\mathrm{Br}}$ and $X(\mathbb{A}_L)^{\mathrm{Br}}$ are related as $L/k$ varies. As a complete description of these sets is generally too subtle, one instead focuses on the vanishing of the Brauer--Manin obstruction over field extensions.

It is known that, for everywhere locally soluble conic bundles, cubic surfaces, and quartic del Pezzo surfaces, the Brauer--Manin obstruction vanishes over every finite extension that is linearly disjoint from a fixed finite extension, provided that the degree of the extension satisfies the appropriate divisibility condition: even degree for conic bundles and quartic del Pezzo surfaces, and degree divisible by three for cubic surfaces. These results follow from \cite[Lemma~3.1]{CV}, \cite[Corollary~1]{swD}, and \cite[Lemma~3.4, Remark~2]{CTPoo}. More generally, an analogous conclusion holds for any variety whose geometric Picard group is torsion-free and geometric Brauer group is trivial, by \cite[Corollary~3.4]{CV}. In this direction, \cite[Theorem~1.1]{RW} shows that for Ch\^atelet surfaces the Brauer--Manin obstruction vanishes over every extension of even degree, provided that either the Brauer group is non-constant or the surface has points over all completions of the base field.

In this paper, we study $(p,mp)$-normic bundles (see Definition~\ref{normic bundle}), which are a higher-dimensional analog of conic bundles and a special case of the notion defined in \cite{AVB} and \cite{AB}. By assuming the existence of local points over the base field, we determine classes of finite extensions $L/k$ for which the Brauer--Manin obstruction vanishes for such varieties. In particular, we obtain a divisibility criterion on the degree $[L:k]$ ensuring that $X(\A_L)^{\Br} \neq \emptyset$. Moreover, assuming Schinzel's hypothesis\cite{ScH}, \(X\) has $L$-rational points by \cite[Theorem~4.2]{CsD}.

\begin{definition}\label{normic bundle}
    Let $p$ be a prime number and $m\geq 1$. A \emph{$(p,mp)$-normic bundle} is a smooth compactification of the affine variety $\Nm_{K/k}(\vec{z})=P(x)$, where $K/k$ is a Galois extension of degree $p$ and $P(x)\in k[x]$ is a separable polynomial of degree $mp$. 

\end{definition}

Note that when $p=2$, a $(p,mp)$-normic bundle is precisely a conic bundle with $2m$ geometric bad fibers, and when \(p=m=2\), it is a Ch\^atelet surface. We now state our main vanishing results.

\begin{theorem}\label{main-vanishing}
Let \(X\) be a \((p,mp)\)-normic bundle over a number field \(k\), and let $P(x)$ be the separable polynomial appearing in the norm equation defining $X$. Suppose that, for every irreducible
factor \(P_i(x)\) of \(P(x)\), the residue field
\(\kappa(P_i)=k[x]/(P_i(x))\) is a Galois extension of \(k\). If
\(X(\A_k)\neq\emptyset\), then the following holds.

\begin{enumerate}
    \item If \(m \leq 2\), then for every finite Galois extension
    \(L/k\) with \(p\mid[L:k]\), we have \( X(\A_L)^{\Br}\neq\emptyset.\)
    Moreover, the same conclusion holds for every extension \(L/k\) with \([L:k]=p\),
    
    \item If \(m\ge3\), then for every finite Galois extension \(L/k\) with
    \(p^{m+1}\mid[L:k]\), we have \(X(\A_L)^{\Br}\neq\emptyset.\) 
    
    Moreover, if \(3\le m\le p\), then for every finite Galois extension \(L/k\)
    with \(p^{m-1}\mid[L:k]\), we have
    \( X(\A_L)^{\Br}\neq\emptyset.\)
\end{enumerate}
\end{theorem}

The next result strengthens Theorem~\ref{main-vanishing} for \(p=2\) and \(3\).
It also removes the assumption that the residue fields of the irreducible
factors of \(P(x)\) are Galois over \(k\). The key point is that the additive
decompositions of \(2p\) are simple in these cases. This allows us
to extend the conclusion to arbitrary finite extensions \(L/k\) whose degree is
divisible by \(p\), without assuming that \(L/k\) is Galois. When \(p=2\), this
partially recovers \cite[Theorem~1.1]{RW} for Châtelet surfaces; when \(p=3\),
it gives the analogous result for \((3,6)\)-normic bundles.
% {\color{blue}Not all conic bundles are normic bundles}.
\begin{theorem}\label{3-6}
Let \(X\) be a \((p,2p)\)-normic bundle over a number field \(k\), where
\(p=2\) or \(3\). Assume that \(X(\A_k)\neq\emptyset\). Then, for every finite
extension \(L/k\) satisfying \(p\mid [L:k]\), one has
\(X(\A_L)^{\Br}\neq\emptyset\).
\end{theorem}

The divisibility hypothesis on \([L:k]\) in Theorem~\ref{main-vanishing} is essential; see Example~\ref{bad place split completely implies persistence}. However, it is not sufficient to assume only that \(p\mid [L:k]\). This is demonstrated by Theorem~\ref{Conic bundle BM}, which shows that the Brauer--Manin obstruction may persist after a finite Galois extension $L/k$ whose degree is divisible by $p$. More precisely, the variety \(X\) is a conic bundle with six geometric bad fibers. By Theorem~\ref{main-vanishing}, we have \(X(\mathbb{A}_L)^{\Br}\neq\emptyset\) for every finite Galois extension \(L/k\) satisfying \(2^4\mid [L:k]\). On the other hand, for the quadratic extension \(L=\Q(\sqrt{17})\), we prove that \(X(\mathbb{A}_L)^{\Br}=\emptyset\). Thus, the hypothesis \(2^4\mid [L:k]\) in Theorem~\ref{main-vanishing} cannot, in general, be weakened to the assumption \(2\mid [L:k]\). Moreover, this example shows that one cannot expect Theorem~\ref{3-6} to hold for arbitrary conic bundles, since the variety in Theorem~\ref{Conic bundle BM} is a \((2,6)\)-normic bundle.

This example also complements the result of \cite[Theorem~1.2]{RW}, which shows that there exists a finite collection of quadratic extensions of $k$ outside of which the Brauer--Manin obstruction vanishes after base change. Our example shows that this exceptional collection may indeed be nonempty.

\begin{theorem}\label{Conic bundle BM}
Let $X$ be the conic bundle surface over $\mathbb{Q}$ defined by the affine equation \[y^2-5z^2=f(x)\] where $ f(x) = -(67x^4 + 340x^3 + 1095x^2 + 1700x + 1675)(x^2 + 1) $. Let $L=\mathbb{Q}(\sqrt{17})$, and
\(
g(x)
=(1+2\sqrt{17})x^2
+5\sqrt{17}\,x
+5(1+2\sqrt{17})\). Then $X(\mathbb{A}_{\mathbb{Q}})\neq\emptyset$ and $X_L$ has Brauer--Manin obstruction given by the element $\mathcal{B}=(5,g(x))\in\operatorname{Br}(X_L)$.
\end{theorem}

 \subsection{Notation} Let $k$ be a field and let $f \in k[x]$ be irreducible. Then $(f)$ is a maximal ideal of $k[x]$, and the residue field at $f$ is
$\kappa(f) = k[x]/(f)$,
a field extension of $k$ generated by a root of $f$.
Let $G$ be a group and let $H \subset G$ be a subgroup. The normaliser of $H$ in $G$, denoted by $N_G(H)$.
    For $X$ a smooth, projective, geometrically integral variety over a number field $k$, we write $\Br X=\operatorname{H}^2_{\text{\rm \'et}}(X,\mathbb{G}_m)$. If $f_Y : Y \to \mathrm{Spec}(k)$ is a $k$-scheme, then $\mathrm{Br}_0(Y) \subset \mathrm{Br}(Y)$ is the image of the pullback map $f_Y^* : \mathrm{Br}(k) \to \mathrm{Br}(Y)$. We define $\overline{\Br}X:=\Br X/\Br_0X$.
    For a global field $k$, we use $\Omega_k$ to denote the set of places of $k$. For a place $v \in \Omega_k$ we use $k_v$ to denote the corresponding completion and for a $k$-scheme $Y$ we set $Y_v := Y_{k_v}$. We use $\mathbb{A}_k$ to denote the adele ring of $k$. For a subgroup $B \subset \mathrm{Br}(Y)$, $Y(\mathbb{A}_k)^B \subset Y(\mathbb{A}_k)$ denotes the set of adelic points orthogonal to $B$, i.e.,
\[
Y(\mathbb{A}_k)^B
=
\{(y_v) \in Y(\mathbb{A}_k) :
\forall \beta \in B,\;
\sum_{v \in \Omega_k} \mathrm{inv}_v\bigl(\beta(y_v)\bigr) = 0\}.
\]
We define $Y(\mathbb{A}_k)^{\mathrm{Br}} := Y(\mathbb{A}_k)^{\mathrm{Br}(Y)}.$ For $a,b \in k_v^\times$, we use $(a,b)_v$ to denote the local Hilbert symbol, and $(a,b)$ to denote the corresponding quaternion algebra class in $\mathrm{Br}(k)$. For a finite field extension $K/k$, $\Nm_{K/k}$ denotes the field norm map from $K$ to $k$.

\smallskip

\begin{acknowledgements}
The authors would like to thank Brendan Creutz for his continuous support during the preparation of this work and for suggesting Theorem~\ref{Conic bundle BM}. They are also grateful to Bianca Viray for helpful comments on an earlier draft. Computational support from Magma is gratefully acknowledged. Finally, the authors acknowledge the support and hospitality of the Lodha Mathematical Sciences Institute (LMSI) during the thematic programme “Rational Points, Algebraic Cycles, and the Local--Global Principle,” where several fruitful discussions related to this work took place. M. Biswas was partially supported by UC Doctoral Scholarship (UCID: 96651612). D. C-Ramachandran was partially supported by DST-INSPIRE Fellowship (Reg No: IF210208). B. Samanta was partially supported by an institute postdoctoral fellowship from IISER Mohali.
\end{acknowledgements}

%%%%%%%%%%%%%%%%%%%%%%%%%%%%%%%%%%%%%%%%%%%%%%%%%%%%%%%%%%%%%%%%%%%%%%%%%%%%%%%%%%%%%%%%%%%%%%%%%%%%%%%%%%%%%%%%%
\section{Good Reduction of Normic Bundles}
In this section, we determine the places of good reduction of the smooth proper compactification of a normic bundle. The main result is Theorem~\ref{places of good reduction for normic bundles}. We begin by recalling some standard facts on norms, traces, and integral normal bases.
\begin{definition}
Let \(B/A\) be a finite free ring extension. The norm \(\Nm_{B/A}:B\to A\) and trace \(\Tr_{B/A}:B\to A\) are defined as the determinant and trace of the \(A\)-linear endomorphism \(x\mapsto bx\).
\end{definition}

It follows immediately from the definition that \(\Nm_{B/A}\) is multiplicative and
\(\Tr_{B/A}\) is additive. Thus, we have group homomorphisms
\(\Nm_{B/A}: B^\times \to A^\times\) and
\(\Tr_{B/A}: B \to A\). Moreover, if \(B_1/A\) and \(B_2/A\) are two ring extensions that are free
\(A\)-modules of finite rank, then
\(\Nm_{B_1\times B_2/A}(x)
= \Nm_{B_1/A}(x_1)\Nm_{B_2/A}(x_2)\)
and
\(\Tr_{B_1\times B_2/A}(x)
= \Tr_{B_1/A}(x_1)+\Tr_{B_2/A}(x_2)\)
for every \(x=(x_1,x_2)\in B_1\times B_2\).

\begin{lemma}\label{normal integral basis}
Let $A$ be a discrete valuation ring with fraction field $k$, maximal ideal
$\mathfrak p$, and residue field $\kappa$. Let $K/k$ be a finite Galois
extension with Galois group $G$, and let $B$ be the integral closure of $A$ in
$K$. Assume that the extension of discrete valuation rings $B/A$ is
unramified (equivalently, $\mathfrak p$ is unramified in $K$). Then
$B\cong A[G]$ as $A[G]$-modules. In particular, $B$ admits an integral normal
basis over $A$.
\end{lemma}

\begin{proof}
Since $\p$ is unramified in $K$, the residue ring $B/\mathfrak{p}B$ is a Galois algebra over $\kappa$ with group $G$. By the normal basis theorem for finite Galois extensions, there exists an element $\bar{\alpha}\in B/\mathfrak{p}B$ whose $G$-orbit forms a $\kappa$-basis. Equivalently, $B/\mathfrak{p}B \cong \kappa[G]\bar{\alpha}$.

Choose any lift $\alpha\in B$, and let $M=A[G]\alpha$. Since the image of $M$ modulo $\mathfrak{p}B$ is all of $B/\mathfrak{p}B$, we have $M+\mathfrak{p}B=B$. By Nakayama's Lemma, it follows that $M=B$.

Now $B$ is a free $A$-module of rank $|G|$, which is also the rank of the group ring $A[G]$. Therefore the surjective $A[G]$-module homomorphism $A[G]\to B$, given by $x\mapsto x\alpha$, is an isomorphism. Hence $B\cong A[G]$ as $A[G]$-modules, so $\alpha$ generates an integral normal basis of $B$ over $A$.
\end{proof}

\begin{lemma}\label{Trace}
If $B=A[G]\alpha$, then $\operatorname{Tr}_{K/k}(\alpha)\in A^\times$.
\end{lemma}

\begin{proof}
Since $\{\sigma(\alpha)\}_{\sigma\in G}$ forms an $A$-basis of $B$, we may write
$1=\sum_{\sigma\in G}c_\sigma\,\sigma(\alpha)$
for uniquely determined coefficients $c_\sigma\in A$.

Applying any $\tau\in G$ to both sides gives
$1=\sum_{\sigma\in G}c_\sigma\,\tau\sigma(\alpha)$.
Replacing $\sigma$ by $\tau^{-1}\sigma$, we obtain
$1=\sum_{\sigma\in G}c_{\tau^{-1}\sigma}\,\sigma(\alpha)$.
By the uniqueness of the basis expansion, $c_{\tau^{-1}\sigma}=c_\sigma$ for every $\sigma,\tau\in G$. Since $G$ acts transitively on itself by left multiplication, all coefficients are equal. Thus there exists $c\in A$ such that $c_\sigma=c$ for every $\sigma\in G$. Consequently,
$1=c\sum_{\sigma\in G}\sigma(\alpha)=c\,\operatorname{Tr}_{K/k}(\alpha)$.
Hence $\operatorname{Tr}_{K/k}(\alpha)=c^{-1}\in A^\times$, showing that the trace of $\alpha$ is a unit of $A$.
\end{proof}
\begin{definition}
Let \(R\) be a discrete valuation ring with fraction field \(K\), and let
\(X\) be a smooth proper \(K\)-variety. An \(R\)-model of \(X\) is an \(R\)-scheme
\(\mathcal X\) together with an isomorphism \(\mathcal X\times_RK\cong X\). We say that \(X\) has \emph{good reduction} if it admits a smooth proper
\(R\)-model.
\end{definition}

Let $K/k$ be a finite separable field extension of degree $n$, and let
$P(x) \in \mathcal{O}_k[x]$ be a separable polynomial of degree $dn$, where
$d$ is a positive integer. Consider the affine norm hypersurface
\[
X_0: \Nm_{K/k}(\vec{z}) = P(x)
\]
in $\mathbb{A}^{n+1}_k$. Following the construction in \cite[Section~2]{AVB}, let $X$ be the smooth proper model of $X_0$.

\begin{theorem}\label{places of good reduction for normic bundles}
    Let $X \to \mathbb{P}^1_k$ be the smooth proper compactification of the affine norm hypersurface $X_0$, defined by $\Nm_{K/k}(\vec{z})=P(x)$, constructed in \cite[Theorem~1.1]{AB}, where $K/k$ is a cyclic extension of degree $p$ and $P(x)\in \mathcal{O}_k[x]$ is a polynomial of degree $mp$. Let $S=\{\mathfrak{p}\in \mathcal{O}_k : v_{\mathfrak{p}}(\Delta_K\Delta(P))\neq 0\}$. Then $X$ has good reduction at every prime $\mathfrak{p}\notin S$.
\end{theorem}
\begin{proof}
Fix a prime \(\mathfrak p\notin S\), and let
\(\mathcal{O}_{k,\mathfrak p}\) be the localization of
\(\mathcal{O}_k\) at \(\mathfrak p\). Let
\(\mathcal{O}_{K,\mathfrak p}\) denote the integral closure of
\(\mathcal{O}_{k,\mathfrak p}\) in \(K\). Consider the affine
\(\mathcal{O}_{k,\mathfrak p}\)-scheme
\[
\mathcal{X}_{0}^{\mathfrak p}:
\Nm_{\mathcal{O}_{K,\mathfrak p}/
\mathcal{O}_{k,\mathfrak p}}(\mathbf z)=P(x).
\]
Since \(\mathfrak p\nmid\Delta_K\), the extension
\(\mathcal{O}_{K,\mathfrak p}/\mathcal{O}_{k,\mathfrak p}\) is unramified. Hence, by Lemma~\ref{normal integral basis},
\(\mathcal{O}_{K,\mathfrak p}\) admits a normal integral basis over
\(\mathcal{O}_{k,\mathfrak p}\), which in particular gives a basis of \(K\) over \(k\). Consequently, the generic fibre of
\(\mathcal{X}_{0}^{\mathfrak p}\) is the original affine normic bundle, \(\mathcal{X}_{0}^{\mathfrak p}\times_{\mathcal{O}_{k,\mathfrak p}}
\operatorname{Spec}k
\cong
X_0\).
The natural projection
\[
\mathcal{X}_{0}^{\mathfrak p}\longrightarrow
\mathbb{A}^{1}_{\mathcal{O}_{k,\mathfrak p}},
\quad
(\mathbf z,x)\longmapsto x,
\]
allows us to apply the construction of \(Y_a\) from
\cite[Section~3]{AB} over the ring
\(\mathcal{O}_{k,\mathfrak p}[x]\), taking \(a=P(x)\).
Since \(\mathfrak p\nmid\Delta(P)\), the divisor
\(V(P(x))\) is smooth over
\(\operatorname{Spec}\mathcal{O}_{k,\mathfrak p}\). Therefore,
\cite[Corollary~3.4]{AB} shows that the corresponding scheme
\(Y_{P(x)}\) is smooth over
\(\operatorname{Spec}\mathcal{O}_{k,\mathfrak p}\).

Next, we carry out the construction of
\cite[Section~4]{AB} with the extension
\(\mathcal{O}_{K,\mathfrak p}/
\mathcal{O}_{k,\mathfrak p}\) in place of \(K/k\).
Since \(\mathfrak p\nmid\Delta_K\),
Lemma~\ref{normal integral basis} provides a normal integral basis of
\(\mathcal{O}_{K,\mathfrak p}\) over
\(\mathcal{O}_{k,\mathfrak p}\). Together with
Lemma~\ref{Trace}, this yields the
\(\mathcal{O}_{K,\mathfrak p}\)-algebra isomorphism
\[
\psi:
\operatorname{Proj}
\mathcal{O}_{K,\mathfrak p}[\{z_v:v\in\mathcal{V}_n\}]
\longrightarrow
\operatorname{Proj}
\mathcal{O}_{K,\mathfrak p}[\{y_v:v\in\mathcal{V}_n\}]
\]
constructed in \cite{AB}. The remaining construction proceeds exactly as in \cite[Section~4]{AB} by faithfully flat descent.

Finally, the compactification argument of
\cite[Lemma~5.1, Proposition~5.2]{AB}
extends over
\(\mathcal{O}_{k,\mathfrak p}\). Gluing
\(Y_{P(x)}\) and
\(Y_{P(1/x')x'^{pd}}\) over
\(\operatorname{Spec}\mathcal{O}_{k,\mathfrak p}[x^{\pm1}]\) and
\(\operatorname{Spec}\mathcal{O}_{k,\mathfrak p}[x'^{\pm1}]\)
produces a smooth proper
\(\mathcal{O}_{k,\mathfrak p}\)-scheme
\(\mathcal X\) whose generic fibre is \(X\), that is,
\[
\mathcal X\times_{\mathcal{O}_{k,\mathfrak p}}
\operatorname{Spec}k
\cong
X.
\]

Thus \(X\) admits a smooth proper model over
\(\mathcal{O}_{k,\mathfrak p}\) for every
\(\mathfrak p\notin S\). Hence \(X\) has good reduction at every prime outside \(S\).
\end{proof}

%%%%%%%%%%%%%%%%%%%%%%%%%%%%%%%%%%%%%%%%%%%%%%%%%%%%%%%%%%%%%%%%%%%%%%%%%%%%%%%%%%%%%%%%%%%%%%%%%%%%%%%%%%%%%%%%%%%%%%%%%%%%%%%%%%%%%%%%
\section{Auxiliary Results}
%%%%%%%%%%%%%%%%%%%%%%%%%%%%%%%%%%%%%%%%%%%%%%%%%%%%%%%%%%%%%%%%%%%%%%%%%%%%%%%%%%%%%%%%%%%%%%%%%%%%%%%%%%%%%%%%%%%%%%%%%%%%%%%%%%%%%%%%
In this section, we record several group-theoretic and field-theoretic lemmas that streamline the proof of Theorem~\ref{main-vanishing}.
%%%%%%%%%%%%%%%%%%%%%%%%%%%%%%%%%%%%%%%%%%%%%%%%%%%%%%%%%%%%%%%%%%%%%%%%%%%%%%%%%%%%%%%%%%%%%%%%%
\subsection{Group-theoretic results}
%%%%%%%%%%%%%%%%%%%%%%%%%%%%%%%%%%%%%%%%%%%%%%%%%%%%%%%%%%%%%%%%%%%%%%%%%%%%%%%%%%%%%%%%%%%%%%%%%%%

\begin{lemma}[Frattini]\label{Frattini's argument}
Let $G$ be a finite group and $H$ be a normal subgroup of $G$. If $P$ is a Sylow $p$-subgroup of $H$, then $G=N_G(P)\,H$.
\end{lemma}
\begin{proof}
See \cite[Theorem~3.1.4]{Zas}.   
\end{proof}

\begin{theorem}[Schur--Zassenhaus]\label{Schur--Zassenhaus}
Let $|G|=ab$ with $(a,b)=1$. If $G$ has a normal subgroup $H$ of order $a$, then $G$ has a subgroup $S$ of order $b$, and $G=HS\cong H\rtimes S$.
\end{theorem}
\begin{proof}
    See \cite[Theorem~6.2.1]{Zas}
\end{proof}
\begin{lemma}[Dedekind's Modular Law]\label{Group theory lemma for compositum}
     Let $G$ be a group with subgroups $A, B$, and $C$ such that $A \leq C$. Then, as sets, $(B \cap C)A = BA \cap C$. Furthermore, if $B$ is normal in $G$, both sides form identical valid subgroups of $G$.
\end{lemma}
\begin{proof}
See \cite[Lemma~1.3.14]{Robinson}.
\end{proof}

\subsection{Field-theoretic results}

\begin{lemma}\label{field-theory lemma}
Let \(L/k\) be a finite Galois extension such that \(p\mid [L:k]\). Let
\(E_1,\dots,E_r\) be Galois subfields of \(L/k\) satisfying
\(p\nmid [E_i:k]\) for each \(i\). Then there exists an intermediate field
\(R\) of \(L/k\) such that \(p\mid [R:k]\) and \(R\cap E_i=k\) for all
\(1\le i\le r\).
\end{lemma}

\begin{proof}
Let \(G=\Gal(L/k)\), and let \(C\) be the compositum of the fields
\(E_1,\dots,E_r\). Since each extension \(E_i/k\) has degree coprime to \(p\),
we have \(p\nmid [C:k]\). Let \(H\trianglelefteq G\) be the subgroup corresponding to \(C\).

Let \(P\) be a Sylow \(p\)-subgroup of \(G\). Since \(p\nmid [G:H]\), we have
\(P\le H\). By Theorem~\ref{Schur--Zassenhaus}, there exists a subgroup
\(S\le N_G(P)\) such that \(N_G(P)=P\rtimes S\). In particular, \(P\cap S=\{1\}\), and hence
\(p\nmid |S|\). Since \(p\mid |G|\), it follows that \(p\mid [G:S]\). As \(P\le H\trianglelefteq G\), Lemma~\ref{Frattini's argument} yields
\(G=N_G(P)H\). Using \(N_G(P)=PS\) and \(P\le H\), we obtain \(G=(PS)H=SH\).

Set \(R=L^S\). Then \(p\mid [R:k]\). For each \(i\), let \(H_i\le G\) be the
subgroup corresponding to \(E_i\). Since \(E_i\subseteq C\), we have
\(H_i\supseteq H\). Therefore,
\(SH_i\supseteq SH=G\), and hence \(SH_i=G\). By the Galois correspondence,
$$
R\cap E_i
=
L^S\cap L^{H_i}
=
L^{SH_i}
=
L^G
=
k.
$$
Thus \(R\cap E_i=k\) for all \(1 \le i \le r\).
\end{proof}

\begin{lemma}\label{Field theory lemma for compositum}
    Let $K, E,$ and $L$ be extensions of a base field $k$ contained within a common algebraic closure. Assume $K/k$ is a Galois extension, and $E \subseteq L$. Then,$$KE \cap L = (K \cap L)E$$
\end{lemma}
\begin{proof}
    Let $G$ be the Galois group of a finite Galois extension containing $K, E,$ and $L$. Let $B, C,$ and $A$ be the subgroups fixing $K, E,$ and $L$, respectively. Since $K/k$ is Galois, $B \triangleleft G$. Since $E \subseteq L$, the Galois correspondence yields $A \leq C$. Using the Galois correspondence and Lemma~\ref{Group theory lemma for compositum}, the result follows.
\end{proof}

\begin{lemma}\label{base-change-factors-prime-to-p}
Let \(P(x)\in k[x]\), and let \(S=\{P_1,\dots,P_r\}\) be the set of irreducible factors of \(P(x)\) over \(k\). Let \(L/k\) be a finite extension, and suppose that each \(\kappa(P_i)/k\) is Galois. Set
$$S_1=\{P_i\in S : p\mid [\kappa(P_i)\cap L:k]\},\,
S_2=\{P_i\in S : p\nmid [\kappa(P_i)\cap L:k]\}.$$ Let
$E=\prod_{P_i\in S_1}(\kappa(P_i)\cap L)$ be the compositum inside \(L\). Then, for every irreducible factor \(Q_j(x)\) of \(P(x)\) over \(E\), one has
$p\nmid [\kappa(Q_j)\cap L:E].$
\end{lemma}
\begin{proof}
Let \(Q_j(x)\) be an irreducible factor of \(P(x)\) over \(E\). Then \(Q_j(x)\)
divides \(P_i(x)\) for some \(P_i\in S\). Let \(\alpha\) be a root of
\(Q_j(x)\). Since \(\kappa(P_i)/k\) is Galois and \(\alpha\) is a root of
\(P_i(x)\), we have \(k(\alpha)=\kappa(P_i)\). Hence \(\kappa(Q_j)=E(\alpha)=E\cdot\kappa(P_i)\). Since \(E\subseteq L\), Lemma~\ref{Field theory lemma for compositum} gives
$$\kappa(Q_j)\cap L=E\cdot\kappa(P_i)\cap L=(\kappa(P_i)\cap L)E.$$

If \(P_i\in S_1\), then \(\kappa(P_i)\cap L\subseteq E\), by the definition of
\(E\). Therefore \(\kappa(Q_j)\cap L=E\), so \([\kappa(Q_j)\cap L:E]=1\). In particular,
\(p\nmid [\kappa(Q_j)\cap L:E]\).

If \(P_i\in S_2\), then \([\kappa(Q_j)\cap L:E]= [(\kappa(P_i)\cap L)E:E]\). This degree divides \([\kappa(P_i)\cap L:k]\). Since \(P_i\in S_2\), we have \(p\nmid [\kappa(P_i)\cap L:k]\). Hence, \(p\nmid [\kappa(Q_j)\cap L:E]\). This proves the claim.
\end{proof}
%%%%%%%%%%%%%%%%%%%%%%%%%%%%%%%%%%%%%%%%%%%%%
\section{Vanishing of Brauer--Manin Obstruction for Normic bundles}
%%%%%%%%%%%%%%%%%%%%%%%%%%%%%%%%%%%%%%%%%%%%%%%%%%%%%%%%%%%%%%%%%%%%%%%%%%%%%%%%%%%%%%%%
In this section, we prove Theorem~\ref{main-vanishing} and Theorem~\ref{3-6} by deducing them from the more general Theorem~\ref{generalized-p-mp}. We begin with the following lemma. It essentially  follows from \cite[Corollary~3.4]{CV}, we include a proof since we require the condition on the field extension to be stated explicitly.

\begin{lemma}\label{isomorphism in Brauer group}
Let $X$ be a $(p,mp)$-normic bundle defined by the affine equation $\Nm_{K/k}(z)=P(x)$, and suppose that $X(\A_k)\neq\emptyset$. Let $L/k$ be a finite extension linearly disjoint from $K$ and from $\kappa(P_i)=k[x]/(P_i(x))$ for every irreducible factor $P_i(x)$ of $P(x)$. If $p\mid [L:k]$, then $X(\A_L)^{\Br}\neq\emptyset$.
\end{lemma}

\begin{proof}
Since \(L\) is linearly disjoint from \(K\) over \(k\), the extension \(LK/L\) is cyclic of degree \(p\). Hence the base change \(X_L\) is the \((p,mp)\)-normic bundle over \(L\) defined by the affine equation \(\Nm_{LK/L}(z)=P(x)\). Let $S_k=\{P_1,\dots,P_r\}$ denote the set of irreducible factors of $P(x)$ in $k[x]$, with $d_i=\deg P_i$. Since $L/k$ is linearly disjoint from $\kappa(P_i)$ for every $i$, we have $L\otimes_k\kappa(P_i)\simeq L[x]/(P_i(x))$, which is a field. Hence each $P_i(x)$ remains irreducible over $L$, so $|S_L|=|S_k|$ and the degrees $d_i$ are preserved. Let
$$
G=\frac{\{(n_i)\in(\mathbb Z/p\mathbb Z)^{|S_k|}:\sum n_i d_i\equiv0\pmod p\}}{\langle(1,\dots,1)\rangle}.
$$
By \cite[Theorem~3.2]{AVB}, there are surjective homomorphisms $\phi_k:G\twoheadrightarrow\Br X/\Br k$ and $\phi_L:G\twoheadrightarrow\Br X_L/\Br L$, given by $(n_i)\mapsto(\chi_K,\prod_iP_i^{n_i})$ and $(n_i)\mapsto(\chi_{LK},\prod_iP_i^{n_i})$, respectively. Since $\phi_L=\res_{L/k}\circ\phi_k$, the restriction map $\res_{L/k}:\Br X/\Br k\to\Br X_L/\Br L$ is surjective.

Choose $(P_v)_v\in X(\A_k)$. For each place $w$ of $L$ above $v$, regard $P_v$ as a point $P_w\in X(L_w)$. Let $\alpha\in\Br X$. By \cite[Theorem~3.2]{AVB}, $\Br X/\Br k$ is $p$-torsion. Therefore,  
$$
\begin{aligned}
\sum_{w\in\Omega_L}\inv_w\alpha_L(P_w)
&=\sum_{v\in\Omega_k}\sum_{w\mid v}\inv_w\res_{L_w/k_v}\alpha(P_v)=[L:k]\sum_{v\in\Omega_k}\inv_v\alpha(P_v)=0,
\end{aligned}
$$
where the second equality follows from \cite[Theorem~1.5.34]{Poo}. The final equality holds in \(\Q/\Z\) since \(\alpha\) is \(p\)-torsion and \(p\mid[L:k]\). Hence $(P_w)_w\in X(\A_L)^{\operatorname{im}(\res_{L/k})}$. Since $\res_{L/k}$ is surjective, it follows that $X(\A_L)^{\Br}\neq\emptyset$.
\end{proof}

\begin{lemma}\label{points below to points above}
Let $X/k$ be a nice variety over a global field $k$, let $L/k$ be a finite extension. If $X(\A_k)^{\Br}\neq\emptyset$, then $X(\A_L)^{\Br}\neq\emptyset$.
\end{lemma}

\begin{proof}
See \cite[Lemma 3.1]{CV}.
\end{proof}

%%%%%%%%%%%%%%%%%%%%%%%%%%%%%%%%%%%%%%%%%%%%%%%%%%%%%%%%%%%%%%%%%%%%%%%%%%%
We now prove a general theorem from which Theorems~\ref{main-vanishing} and~\ref{3-6} follow as immediate consequences. Let \(P(x)\) be the separable polynomial appearing in the norm equation defining \(X\). Since the sum of the degrees of its irreducible factors is at most \(mp\), we introduce the combinatorial quantity \(N(m,p)\), which bounds the largest possible \(p\)-adic valuation of the degree of the compositum of the fields \(\kappa(P_i)\cap L\) over all possible factorizations of \(P(x)\). This bound ensures that the remaining relative extension has degree divisible by \(p\), allowing us to apply Lemma~\ref{isomorphism in Brauer group}.

\begin{definition}\label{def:combinatorial-exponent}
Let $m\ge 2$ be an integer and let $p$ be a prime. Define
$$
N(m,p)=
\max\left\{
\sum_{i=1}^r\bigl(1+v_p(c_i)\bigr)
\;\middle|\;
\sum_{i=1}^r c_i\le m,\;
c_i\ge 2
\text{ for }1\le i\le r
\right\},
$$
where $v_p$ denotes the $p$-adic valuation.
\end{definition}

\begin{theorem}\label{generalized-p-mp}
    Let $X$ be a $(p, mp)$-normic bundle over a number field $k$, where $m \geq 2$, and let $P(x)$ be the separable polynomial appearing in the norm equation defining $X$. Suppose that, for every irreducible factor \(P_i(x)\) of \(P(x)\), the residue field
    \(\kappa(P_i)=k[x]/(P_i(x))\) is a Galois extension of \(k\). If $X(\A_k) \neq \emptyset$, then for every finite Galois extension $L/k$ with $p^{N(m,p) + 1} \mid [L:k]$, we have \(X(\A_L)^{\Br} \neq \emptyset\).
\end{theorem}

\begin{proof}
Let $L/k$ be a finite Galois extension with $p^{N(m,p)+1} \mid [L:k]$. If $L$ is linearly disjoint from $K$ and from all $\kappa(P_i)$, we are done by Lemma~\ref{isomorphism in Brauer group}. Suppose $L$ is not linearly disjoint from $K$. Since $K/k$ is a Galois extension of prime degree $p$, we get $K\subset L$. This implies $X(L)\neq\emptyset$ and hence $X(\A_L)^{\Br}\neq\emptyset$. If $P(x)$ is irreducible, then $\overline{\Br}X=0$ by \cite[Corollary 3.3]{AVB} and we have $X(\A_k)^{\Br}=X(\A_k)\neq\emptyset$. By Lemma~\ref{points below to points above}, we have $X(\A_L)^{\Br}\neq\emptyset$ for all $L/k$.

We may now assume that $P(x)$ is reducible and that $L$ is linearly disjoint from $K$, but not from all of the fields $\kappa(P_i)$. Let $S=\{P_1,\dots,P_r\}$ denote the set of irreducible factors of $P(x)$ over $k$. Consider the partition
$$
S_1=\{P_i\in S : p\mid [\kappa(P_i)\cap L:k]\}
\quad \text{and} \quad
S_2=\{P_i\in S : p\nmid [\kappa(P_i)\cap L:k]\}.
$$
We distinguish the following cases.

\noindent\textit{Case 1.} Suppose that $S_1=\emptyset$, that is, $p\nmid [\kappa(P_i)\cap L:k]$ for every $P_i\in S$. 

Applying Lemma~\ref{field-theory lemma} to $E_i=\kappa(P_i) \cap L$, we get a subextension $R/k$ of $L/k$ such that $p \mid [R:k]$ and $R \cap (\kappa(P_i)\cap L)=k$ for all $i$. Therefore, $R$ is linearly disjoint from $\kappa(P_i)$ for all $i$. By Lemma~\ref{isomorphism in Brauer group}, we have $X(\A_R)^{\Br}\neq\emptyset$. Applying Lemma~\ref{points below to points above}, we conclude that $X(\A_L)^{\Br}\neq\emptyset$.

\noindent\textit{Case 2.} Suppose that $S_1\neq\emptyset$, that is, $p\mid [\kappa(P_i)\cap L:k]$ for some $P_i\in S_1$.

We divide the argument into two subcases. 

\smallskip

\noindent\textit{Subcase 2A.} There exists \(P_i\in S_1\) with
\(\deg(P_i)=p\).

Since \(\kappa(P_i)\cap L\subseteq \kappa(P_i)\) and
\([\kappa(P_i):k]=p\), the divisibility condition \(p\mid [\kappa(P_i)\cap L:k]\)
forces \([\kappa(P_i)\cap L:k]=p\). Hence \(\kappa(P_i)\cap L=\kappa(P_i)\), and therefore \(\kappa(P_i)\subseteq L\). It follows that \(X(L)\neq\emptyset\), and hence
\(X(\A_L)^{\Br}\neq\emptyset\).

When \(m=2\), this is the only possible subcase. Indeed, if \(P(x)\) is
reducible and some factor \(P_i\) satisfies \(p\mid [\kappa(P_i)\cap L:k]\),
then \(p\mid \deg(P_i)\). Since the degrees of the irreducible factors of
\(P(x)\) sum to \(2p\), there must be an irreducible factor of degree \(p\).
Thus Subcase~2A applies. Consequently, for \(m=2\), whenever \(p\mid [L:k]\), we obtain
\(X(\A_L)^{\Br}\neq\emptyset\). Since \(N(2,p)\ge 0\), any Galois extension \(L/k\) satisfying \(p^{N(2,p)+1}\mid [L:k]\) also satisfies \(p\mid [L:k]\). Hence the desired conclusion follows for \(m=2\).

We may therefore assume from now on that \(m\ge 3\) and that Subcase~2A does
not occur.

\smallskip

\noindent\textit{Subcase 2B.} For every \(P_i\in S_1\), \(\deg(P_i)=c_i p\) for some integer \(c_i\ge 2\).
For each $P_i \in S_1$, define the intersection field $E_i = \kappa(P_i) \cap L$. Because $\kappa(P_i)/k$ and $L/k$ are both Galois extensions, $E_i/k$ is a Galois extension. Since $E_i \subseteq \kappa(P_i)$, the degree $[E_i:k]$ must divide $[\kappa(P_i):k] = c_i p$. Therefore, its $p$-adic valuation satisfies $v_p([E_i:k]) \le 1 + v_p(c_i)$.

Let $E$ be the compositum of all $E_i$ for $P_i \in S_1$. Since each $E_i/k$ is Galois, $E/k$ is also Galois. The $p$-adic valuation of the compositum is bounded by the sum of the valuations of the individual extensions:
$$
v_p([E:k]) \le \sum_{P_i \in S_1} v_p([E_i:k]) \le \sum_{P_i \in S_1} (1 + v_p(c_i)).
$$
Because the total degree of $P(x)$ is $mp$, we have $\sum_{P_i \in S_1} c_i \le m$. By Definition~\ref{def:combinatorial-exponent}, this sum is bounded above by the combinatorial exponent $N(m, p)$. Thus, $v_p([E:k]) \le N(m, p)$. 
By hypothesis, $p^{N(m,p)+1} \mid [L:k]$, which means $v_p([L:k]) \ge N(m, p) + 1$. Therefore,
$$
v_p([L:E]) = v_p([L:k]) - v_p([E:k]) \ge 1,
$$
meaning $p \mid [L:E]$.

We now analyze the irreducible factors of \(P(x)\) over the field \(E\). Let
\(\{Q_j(x)\}\) denote the set of irreducible factors of \(P(x)\) over \(E\).
By Lemma~\ref{base-change-factors-prime-to-p}, each \(Q_j(x)\) satisfies
\(p\nmid [\kappa(Q_j)\cap L:E]\). Thus, after replacing the base field \(k\) by
\(E\), we are in the situation of Case~1 for the extension \(L/E\). Hence, $X(\mathbb{A}_L)^{\Br}\neq\emptyset$.
\end{proof}\begin{lemma}\label{Nmp-bound}
Let $m\ge3$ be an integer and let $p$ be a prime. If $m<p$, then
\(N(m,p)\le m-2.\)
\end{lemma}

\begin{proof}
Since $m<p$, any $c_i$ satisfies $c_i<p$, have $v_p(c_i)=0$. Therefore,
$$
N(m,p)=\max\{r:\sum_{i=1}^r c_i\le m,\ c_i\ge2\}.
$$
As $c_i\ge2$ for every $i$, we have
\(2r\le\sum_{i=1}^r c_i\le m,\) so that
\(r\le\Bigl\lfloor\frac m2\Bigr\rfloor.\) If $m=3$, then $r\le1=m-2$. If $m\ge4$, then
$$
\Bigl\lfloor\frac m2\Bigr\rfloor\le\frac m2\le m-2.
$$ Hence $N(m,p)\le m-2$.
\end{proof}
We now derive Theorem~\ref{main-vanishing} and Theorem~\ref{3-6} from the preceding theorem. The proofs of all our vanishing results for normic bundles rely on the same underlying idea.

\begin{proof}[Proof of Theorem~\ref{main-vanishing}]

(a) Suppose that $L/k$ is Galois with $p\mid[L:k]$. By the same argument as in the proof of Theorem~\ref{generalized-p-mp}, it suffices to consider the case where $P(x)$ is reducible, $L$ is linearly disjoint from $K$, but not from all of the fields $\kappa(P_i)$. Then the result follows from Case~1 and Case~2 (Subcase~2A), in the proof of Theorem~\ref{generalized-p-mp}.

Let \(L/k\) be an extension with \([L:k]=p\). By the same argument as in the
proof of Theorem~\ref{generalized-p-mp}, it suffices to consider the case where
\(P(x)\) is reducible, \(L\) is linearly disjoint from \(K\), but not linearly
disjoint from all of the fields \(\kappa(P_i)\). If \(m=1\), this case cannot
occur, and hence we are done. We may therefore assume that \(m=2\). Since each extension \(\kappa(P_i)/k\) is Galois, if \(L\) is not linearly
disjoint from \(\kappa(P_i)\), then \(L\cap \kappa(P_i)\neq k\). As
\([L:k]=p\) is prime, it follows that \(L\cap \kappa(P_i)=L\), and hence
\(L\subseteq \kappa(P_i)\). Since \(P(x)\) is reducible of degree \(2p\), \(\deg(P_i)<2p\). Moreover, \(L\subseteq \kappa(P_i)\) implies that \(p\mid \deg(P_i)\), so we
must have \(\deg(P_i)=p\). Therefore \(L=\kappa(P_i)\), and hence
\(X(L)\neq\emptyset\).

(b) Let $L/k$ be a finite Galois extension with $p^{m+1} \mid [L:k]$. Then by definition we know that $N(m, p) \le m$. Our assumption $p^{m+1} \mid [L:k]$ implies that $p^{N(m,p)+1} \mid [L:k]$. Then Theorem~\ref{generalized-p-mp} shows that $X(\A_L)^{\Br} \neq \emptyset$.

Let $L/k$ be a finite Galois extension such that $p^{m-1} \mid [L:k]$. If $3 \le m < p$, Lemma~\ref{Nmp-bound} gives $N(m, p) \le m-2$. Consequently, our hypothesis $p^{m-1} \mid [L:k]$ implies $p^{N(m,p)+1} \mid [L:k]$, and the result follows immediately from Theorem~\ref{generalized-p-mp}. 

It remains to consider the case $m = p \ge 3$. If $p \ge 5$, the combinatorial exponent is $N(p,p) = \max\left\{2, \lfloor p/2 \rfloor\right\} = \frac{p-1}{2}$. Since $\frac{p-1}{2} \le p-2$ for all $p \ge 5$, we again have $N(p,p) + 1 \le p - 1$. Our hypothesis ensures $p^{N(p,p)+1} \mid [L:k]$, allowing Theorem~\ref{generalized-p-mp} to apply directly. Finally, suppose $m = p = 3$. If $P(x)$ is irreducible, the case which is resolved at the beginning of the proof of Theorem~\ref{generalized-p-mp}. For the remaining reducible scenario (Subcase 2B of Theorem~\ref{generalized-p-mp}),  the requirement that $c_i < 3$ combined with the condition $c_i \ge 2$ forces $c_i = 2$. The valuation sum bounding $v_3([E:k])$ is strictly bounded by $1 + v_3(2) = 1$. This effective bound is perfectly satisfied by our hypothesis that $3^2 \mid [L:k]$, which ensures $v_3([L:E]) \ge 2 - 1 = 1$. Thus, the divisibility condition $3 \mid [L:E]$ of Subcase 2B holds, yielding $X(\A_L)^{\Br} \neq \emptyset$.\end{proof}

Before proving Theorem~\ref{3-6}, we isolate the following proposition, which is the key ingredient in the argument.

\begin{proposition}\label{quadratic-cubic-factors}
Let \(X\) be a \((p,mp)\)-normic bundle over a number field \(k\). Suppose that
every irreducible factor of \(P(x)\) has degree \(2\) or \(3\). If
\(X(\mathbb{A}_k)\neq\emptyset\), then for every finite extension \(L/k\)
satisfying \(p\mid [L:k]\), one has
\(X(\mathbb{A}_L)^{\Br}\neq\emptyset\).
\end{proposition}

\begin{proof}
By Lemma~\ref{isomorphism in Brauer group}, it remains only to consider the
case where \(L\) is not linearly disjoint from at least one of the fields
\(\kappa(P_i)\). For such an \(i\), the tensor product
\(L\otimes_k\kappa(P_i)\) is not a field. Since \(\deg P_i\leq 3\), this forces
\(P_i\) to have a linear factor over \(L\). Hence \(P(x)\) has a root over
\(L\), so \(X(L)\neq\emptyset\). Therefore
\(X(\mathbb{A}_L)^{\Br}\neq\emptyset\).
\end{proof}

\begin{proof}[Proof of Theorem~\ref{3-6}]
Fix a prime $p\in\{2,3\}$, $L/k$ be an extension with $p\mid[L:k]$. By the same argument as in the proof of Theorem~\ref{generalized-p-mp}, it suffices to consider the case where $P(x)$ is reducible, $L$ is linearly disjoint from $K$, but not from all of the fields $\kappa(P_i)$. Moreover, if \(P(x)\) has a linear factor over \(k\), then \(X(k)\neq\emptyset\). It follows that \(X(\mathbb{A}_k)^{\Br}\neq\emptyset\), and hence, by Lemma~\ref{points below to points above}, we have \(X(\mathbb{A}_L)^{\Br}\neq\emptyset\) for every extension \(L/k\). Since \(\deg P(x)= 4\, \text{or}\,6\), there are only two cases to consider.

\noindent\emph{Case 1.} Every irreducible factor of $P(x)$ has degree $2$ or $3$. Hence by Proposition~\ref{quadratic-cubic-factors} we have $X(\A_L)^{\Br}\neq\emptyset$.

\noindent\textit{Case 2.} Suppose \(P(x)=P_1(x)P_2(x)\), where \(\deg P_1=2\) and \(\deg P_2=4\). Therefore $p=3$ in this case. Then, by \cite[Corollary 3.3]{AVB}, we have \(\overline{\Br} X=0\). Hence \(X(\mathbb{A}_k)^{\Br}=X(\mathbb{A}_k)\neq\emptyset\). By Lemma~\ref{points below to points above}, it follows that \(X(\mathbb{A}_L)^{\Br}\neq\emptyset\).
\end{proof}

%%%%%%%%%%%%%%%%%%%%%%%%%%%%%%%%%%%%%%%%%%%%%%%%%%%%%%%%%%%%%%%%%%%%%%%%%%%%%%%%%%%%%%%%%%%%%%%%%%%%%%%%%%%%%%%%%%%%%%%%%%%%%%%%%%%%%%%%%%%%%%%%%%%%%%%%%%%%%%%%%%%%%%%%
\section{Persistence of Brauer--Manin Obstruction for Normic Bundles}
%%%%%%%%%%%%%%%%%%%%%%%%%%%%%%%%%%%%%%%%%%%%%%%%%%%%%%%%%%%%%%%%%%%%%%%%%%%%%%%%%%%%%%%%%%%%%%%%%%%%%%%%%%%%%%%%%%%%%%%%%%%%%%%%%%%%%%%%%%%%%%%%%%%%%%%%%%%%%%%%%%%%%%%%
In this section, we prove Theorem~\ref{Conic bundle BM}, which shows that the hypotheses in Theorem~\ref{main-vanishing} cannot in general be weakened to a smaller degree, and that Theorem~\ref{3-6} does not extend to arbitrary conic bundles. We begin with a lemma showing that conic bundles are locally soluble at all sufficiently large primes.

\begin{lemma}\label{local solubility}
Let \(X\) be the surface over \(\mathbb{Q}\) defined by \(y^2-az^2=P(x)\), where \(a\in\mathbb{Z}\setminus\{0\}\) and \(P(x)\in\mathbb{Z}[x]\) is a polynomial of degree \(r\). Let \(p\) be a prime satisfying \(p>r\), \(p\neq 2\), \(p\nmid a\), and such that \(P(x)\) is nonzero modulo \(p\). Then \(X(\mathbb{Q}_p)\neq\emptyset\).
\end{lemma}

\begin{proof}
Since \(P(x)\) has degree \(r\) and is nonzero modulo \(p\), its reduction modulo \(p\) has at most \(r\) roots in \(\mathbb{F}_p\). As \(p>r\), there exists \(x_0\in\mathbb{F}_p\) such that
\(P(x_0)\not\equiv0\pmod p\). Fix such an \(x_0\), and consider the subsets
\[
S_1=\{y^2:y\in\mathbb{F}_p\}
\quad\text{and}\quad
S_2=\{az^2+P(x_0):z\in\mathbb{F}_p\}
\]
of \(\mathbb{F}_p\). Since \(p\neq2\) and \(p\nmid a\), both \(S_1\) and \(S_2\) have cardinality \((p+1)/2\). Therefore, \(|S_1|+|S_2|=p+1>|\mathbb{F}_p|\),
so \(S_1\cap S_2\neq\emptyset\). Hence there exist \(y_0,z_0\in\mathbb{F}_p\) satisfying \(y_0^2\equiv az_0^2+P(x_0)\pmod p\), and thus \((x_0,y_0,z_0)\) is an \(\mathbb{F}_p\)-point of \(X\).

It remains to check that this point is smooth. Let
\(F(x,y,z)=y^2-az^2-P(x)\). At \((x_0,y_0,z_0)\),
\[
\frac{\partial F}{\partial y}=2y_0
\quad\text{and}\quad
\frac{\partial F}{\partial z}=-2az_0.
\]
Since \(p\neq2\) and \(p\nmid a\), both partial derivatives vanish modulo \(p\) only if \(y_0\equiv z_0\equiv0\pmod p\). This would imply \(P(x_0)\equiv y_0^2-az_0^2\equiv0\pmod p\), contradicting the choice of \(x_0\). Hence \((x_0,y_0,z_0)\) is a smooth \(\mathbb{F}_p\)-point of \(X\). By Hensel's lemma, it lifts to a point of \(X(\mathbb{Q}_p)\).
\end{proof}
Before proving Theorem~\ref{Conic bundle BM}, we recall the notation. $f(x)=-(67x^4+340x^3+1095x^2+1700x+1675)(x^2+1)$ and we have
$g(x)=(1+2\sqrt{17})x^2 +5\sqrt{17}x
+5(1+2\sqrt{17})$,
and let \(X/\mathbb{Q}\) and \(L=\mathbb{Q}(\sqrt{17})\) be as in Theorem~\ref{Conic bundle BM}. 

\begin{proof}[Proof of Theorem~\ref{Conic bundle BM}]
We first show that \(X(\A_{\Q})\neq\emptyset\), and then compute the local invariants of the Brauer class \(\mathcal{B}=(5,g(x))\) over \(L\).

Since \(5>0\), it is a square in \(\mathbb{R}\), and hence
\(X(\mathbb{R})\neq\emptyset\). For every prime \(p\ge7\),
Lemma~\ref{local solubility} gives
\(X(\mathbb{Q}_p)\neq\emptyset\). It therefore remains to consider
\(p=2,3,5\).
For \(p=2,3\), evaluating at \(x=0\) gives
\(f(0)=-1675\), and
\((5,f(0))_p=+1\). Hence the fiber above
\(x=0\) has a \(\mathbb{Q}_p\)-point. For \(p=5\), evaluating at \(x=1\) gives
\(f(1)=-9754\), and
\((5,-9754)_5=+1\). Hence the fiber above
\(x=1\) has a \(\mathbb{Q}_5\)-point.
Thus \(X(\mathbb{A}_{\mathbb{Q}})\neq\emptyset\).

We now compute the local invariants of
\(\mathcal{B}=(5,g(x))\)
at every place \(w\) of \(L\). If \(w\) is archimedean, then
\(5\) is a square, so
\(\operatorname{inv}_w(\mathcal{B}(P_w))=0\). It therefore remains to consider the non-archimedean places.

Let \(S=\{2,5,17,19,89,2268209\}\). We first consider places \(w\nmid p\) for every \(p\in S\) and show that for such places,
\(\operatorname{inv}_w(\mathcal{B}(P_w))=0\) for all \(P_w\in X(L_w)\). 
By Theorem~\ref{places of good reduction for normic bundles},
\(X_L\) has good reduction at \(w\), and
\cite[Theorem~13.3.15]{CS}
shows that the evaluation map is constant. It therefore suffices to compute
\(\operatorname{inv}_w(\mathcal{B}(P_w))\)
at a single point of
\(X(L_w)\).\\
If \(w\nmid67\), then \((-5,-67)_w=+1\), thus \(X_\infty(L_w) \neq \emptyset\). For \(P_w\in X_\infty(L_w)\),  \(\operatorname{inv}_w(\mathcal{B}(P_w))=\operatorname{inv}_w(5,1+2\sqrt{17})\). Since both \(5\) and \(1+2\sqrt{17}\) are \(w\)-adic units, the Hilbert symbol \((5,1+2\sqrt{17})_w=+1\), and hence \(\operatorname{inv}_w(\mathcal{B}(P_w))=0\).\\ 
If \(w\mid67\), then \((-5,f(1))_w =  (-5, -9754)_w=+1\), thus \(X_1(L_w) \neq \emptyset\). For \(P_w\in X_1(L_w)\), \(\operatorname{inv}_w(\mathcal{B}(P_w))= \operatorname{inv}_w(5,6+17\sqrt{17})\). Again, both entries are \(w\)-adic units, so \((5,6+17\sqrt{17})_w=+1\), and therefore \(\operatorname{inv}_w(\mathcal{B}(P_w))=0\).

Next suppose that \(w\mid2\) or \(w\mid17\). Again,
\(\operatorname{inv}_w(\mathcal{B}(P_w))=0\) for all
\(P_w\in X(L_w)\). Indeed,
\(L_w(\sqrt5)/L_w\) is unramified, and
\(g(x)\) has no roots in either \(\mathbb{F}_2\) or
\(\mathbb{F}_{17}\). Since \(g(x)\) has degree \(2\) and
\(w\nmid67\), its valuation is even everywhere, being \(0\) when
\(v_w(x)\ge0\) and \(-2k\) when \(v_w(x)<0\). Hence
\(\operatorname{inv}_w(\mathcal{B}(P_w))=0\).

Now let \(w\mid p\) with
\(p\in\{19,89,2268209\}\). Again,
\(\operatorname{inv}_w(\mathcal{B}(P_w))=0\) for all
\(P_w\in X(L_w)\). Since \(p\) splits completely in \(L\), we have
\(L_w\cong\mathbb{Q}_p\). By quadratic reciprocity,
\(5\) is a square in \(\mathbb{Q}_p\), so \((5,g(x))_w=+1\) for every \(x\in X(L_w)\).

Finally, suppose that \(w\mid5\). We show that \(\operatorname{inv}_w(\mathcal{B}(P_w))
=1/2\) for every \(P_w\in X(L_w)\).\\
Fix \(P_w=[y:z:x]\in X(L_w)\). Since \(5\) is inert in \(L\), the residue field of \(L_w\) is \(\mathbb{F}_{25}\). Note that in this case \(g(x)\equiv(1+2\sqrt{17})x^2\mod w\). We distinguish three cases.

\noindent
\textit{Case 1.} Suppose that \(v_w(x)=0\). Then \(g(x)\equiv(1+2\sqrt{17})x^2\pmod w\), implies \(v_w(g(x))=0\). By \cite[Proposition~3.4]{Neu},
\[
(5,g(x))_w
=
(5,1+2\sqrt{17})_w
=
(1+2\sqrt{17})^{(25-1)/2}\pmod w.
\]
Checking that \(1+2\sqrt{17}\) is a nonsquare in \(\mathbb{F}_{25}\) is equivalent to checking that
\(\operatorname{N}_{\mathbb{F}_{25}/\mathbb{F}_5}(1+2\sqrt{17})\)
is a nonsquare in \(\mathbb{F}_5\). Since \(\operatorname{N}_{\mathbb{F}_{25}/\mathbb{F}_5}(1+2\sqrt{17})
\equiv3\pmod5\), and \(3\) is not a square in \(\mathbb{F}_5\), we obtain \((5,g(x))_w=-1\).

\noindent
\textit{Case 2.} Suppose that \(v_w(x)\ge1\). Then
\(v_w(g(x))\ge1\), and
\[
\begin{aligned}
(5,g(x))_w
&=(5,5)_w(5,g(x)/5)_w\\
&=(5,g(x)/5)_w
\\
&=(g(x)/5)^{(25-1)/2}\mod w\\
&=(1+2\sqrt{17})^{(25-1)/2}\mod w\\
&=-1.
\end{aligned}
\]

\noindent
\textit{Case 3.} Suppose that \(v_w(x)=-k<0\). Then \(g(x)
=
x^2(1+2\sqrt{17})(1+\varepsilon)\), where \(\varepsilon
=
\frac{5\sqrt{17}x+5(1+2\sqrt{17})}
{(1+2\sqrt{17})x^2}\). Since \(v_w(\varepsilon)>0\),
\[
(5,g(x))_w
=
(5,x^2)_w
(5,1+2\sqrt{17})_w
(5,1+\varepsilon)_w.
\]
Now \(1+\varepsilon\) is a square modulo \(w\), and hence, by Hensel's lemma, a square in \(L_w\). Thus
\[
(5,1+\varepsilon)_w=(5,x^2)_w=+1,
\]
while, by Case~1, \((5,1+2\sqrt{17})_w=-1\). Therefore, \((5,g(x))_w=-1\). \\
Hence \(\operatorname{inv}_w(\mathcal{B}(P_w))
=\frac12\) for every \(P_w\in X(L_w)\).

Summing the local invariants over all places of \(L\), we obtain \(\sum_{w\in\Omega_L}
\operatorname{inv}_w(\mathcal{B}(Q_w))
=
\frac12\) for every \((Q_w)_w\in X(\mathbb{A}_L)\). Therefore, \(X(\mathbb{A}_L)^{\mathcal{B}}
=\emptyset\).
\end{proof}

The following example shows that the \(p\)-divisibility hypothesis on \([L:k]\) in Theorem~\ref{main-vanishing} and Theorem~\ref{3-6} is necessary.

\begin{example}\label{bad place split completely implies persistence}
Let \(p\) be a prime, and let \(k\) be a number field containing primitive $p$th root of unity. Let \(X\) be the \((p,2p)\)-normic bundle constructed in \cite[Proposition~4.1]{AVB}. The authors show that \(X\) has Brauer--Manin obstruction to the Hasse principle over \(k\), captured by a cyclic algebra of degree \(p\) representing a class \(\alpha\in\overline{\Br}X\) and has constant evaluation at every place. Let \(S\) denote the places of bad reduction of \(X\), which is a finite set. For each place of \(k\) outside \(S\), the evaluation map is constant by \cite[Theorem 13.3.15]{CS} and hence zero by computation. Let \(n\) be an integer coprime to \(p\). By the Grunwald--Wang theorem, there exists a finite extension \(L/k\) of degree \(n\) in which every place of \(S\) splits completely. It then follows that the Brauer--Manin obstruction on \(X\) persists after base change to \(L\); that is, \(X(\A_L)^{\Br}=\emptyset\).

Thus, we obtain examples of \((p,mp)\)-normic bundles \(X\) and finite extensions \(L/k\) of every degree \(n\) coprime to \(p\) such that
\(X(\mathbb{A}_L)^{\Br}=\emptyset\).
\end{example}

%%%%%%%%%%%%%%%%%%%%%%%%%%%%%%%%%%%%%%%%%%%%%%%%%%%%%%%%%%%%%%%%%%%%%%%%%%%%%%%%%%%%%%%%%%%%%%%%%%%%%%%%%%%%%%%%%%%%%%%%%%%%%%%%%%%%%%%%%%%%%%%%%%%%%%%%%%%%%%%%%%%%%%%%%%%%%%%%%%%%%%%%%%%%%

\end{document}